\documentclass[11pt]{article}
\usepackage{amsmath,amsthm,amsfonts,amssymb,amscd, amsxtra, mathrsfs,textcomp,verbatim,inputenc}
\usepackage{url}
\usepackage[margin=2.5 cm,nohead]{geometry}
\usepackage{color}
\usepackage{float,mathrsfs,threeparttable,color,soul,colortbl,hhline,rotating,array,multirow,makecell}
\usepackage[ruled]{algorithm2e} 

\newtheorem{theorem}{Theorem}
\newtheorem{lemma}[theorem]{Lemma}
\newtheorem{definition}{Definition}

\newtheorem{proposition}[theorem]{Proposition}

\newtheorem{remark}{Remark}

\newcommand{\R}{\mathbb{R}}

\newtheorem{Rule}{Rule}

\begin{document}
\title{Subgradient method with feasible inexact projections for constrained  convex optimization problems}
\author{
A. A. Aguiar \thanks{Instituto de Matem\'atica e Estat\'istica, Universidade Federal de Goi\'as,  CEP 74001-970 - Goi\^ania, GO, Brazil, E-mails: {\tt  ademiraguia@gmail.com},  {\tt  orizon@ufg.br},  {\tt  lfprudente@ufg.br}. The authors was supported in part by  CNPq grants 305158/2014-7 and 302473/2017-3,  FAPEG/PRONEM- 201710267000532 and CAPES.}
\and
O.  P. Ferreira  \footnotemark[1]
\and
L. F. Prudente \footnotemark[1]
}
\maketitle

\maketitle
\begin{abstract}
	In this paper, we propose a new  inexact version  of the projected subgradient method to solve nondifferentiable  constrained  convex optimization problems. The   method combine  $\epsilon$-subgradient  method with a procedure to obtain a feasible inexact projection onto the constraint set.   Asymptotic convergence results   and  iteration-complexity bounds  for the sequence generated by the  method employing the well known  exogenous stepsizes, Polyak's stepsizes, and dynamic stepsizes are  stablished. 
\end{abstract}



\section{Introduction}
The Subgradient method is one of the most interesting  iterative method for solving nondifferentiable convex optimization  problems,  which has its origin and development in the 60's,  see  \cite{Ermolev1966, Shor1985}. Since then, the subgradient method has attracted the attention of the scientific community working on optimization. One of the factors that explains this interest is its simplicity and ease of implementation. In particular, allowing a low cost of storage and ready exploitation of separability and sparsity.  For these reasons, several variants of this method have emerged and properties of it have been discovered throughout the years, resulting in a wide literature on the subject; see, for exemple \cite{AlberIusemSolodov1998, Yunier2013, Bertsekas1999, GoffinKiwiel1999, KiwielBook1985, NedicBertsekas2010} and the references therein.

The aim of this paper is to  present   an inexact version  of the projected subgradient method, which   consists in using an inexact projection instead of the exact one,  for 
minimizing a convex function $f: \mathbb{R}^n \to \mathbb{R}$ onto a   closed and convex subset  $C$ of $\mathbb{R}^n$. The proposed method,  that we call {\it Subgradient-InexP method},  generates a sequence $\{x_k\}$ where each iteration   consists of two stages.  The first stage   performs  a step from the current iterate  $x_k$ in the opposite direction of a $\epsilon$-subgradient of $f$ at $ x_k $ and the second  inexactly projects  the resulting vector onto the feasible set $C$.   From the theoretical point of view,  considering methods that use inexact projections are particularly interesting for  the following reasons.  Even when the projection onto  a convex set  is an easy problem, iterative  methods provide only  approximated solutions with small errors, due to round-off errors in floating-point arithmetics.   Therefore,  the study of inexact methods gives  theoretical support for real computational implementations of exact  schemes.  On the other hand, in general,  one drawback of  methods that use exact projections is having to solve a quadratic problem at each stage,  which may  substantially increasing the cost per iteration  if the number of unknowns is large.   In fact, it may not be justified to compute exact projection when the current iterate $x_k$  is far from the solution of the problem in consideration. Moreover,  a procedure for  computing  a feasible inexact projection may present a low computation cost per iteration in comparison with one that computes the exact projection.  Thus,   it seems reasonable to  consider versions of projected subgradient method that compute the projection only approximately.  In order to present  formally  and analyze the Subgradient-InexP method, we use the concept of feasible inexact projection  with relative error, which was  appeared  in \cite{VillaSalzBaldassarre2013} (see also\cite{OrizonFabianaGilson2018}).  It is worth noting that the concept of feasible inexact projection  also accepts an exact projection when it is easy to obtain.  For instance, the exact projections onto a box  or a  second order cone is very easy to obtain;  see, respectively, \cite[p. 520]{NocedalWright2006} and \cite[Proposition 3.3]{FukushimaTseng2002}. A feasible inexact projection onto a polyhedral closed convex set can be obtained using quadratic programming methods that generate feasible iterates, such as feasible active set methods and interior point methods; see, for example, \cite{NicholasPhilippe2002, NocedalWright2006, Robert1996}. It is worth mentioning that, if the exact projection is used, then Subgradient-InexP method becomes the projected subgradient method considered in \cite{AlberIusemSolodov1998}. Several methods similar to the projected subgradient method have been studied in different papers, see \cite{GoffinKiwiel1999, Mainge2008}. However, as far as we know, none of them use the concept of feasible inexact projection.

The main tool used in our analysis of Subgradient-InexP method is a version of the inequality obtained in \cite[Lemma~1.1]{CorreaLemarecha1993}; see also a variant of it in \cite[Lemma~2.1]{nedic_bertsekas2001}.  By using this inequality, we establish asymptotic convergence results and iteration-complexity bounds for the sequence generated by our method employing the well known exogenous stepsizes, Polyak's stepsizes, and dynamic stepsizes.  We point out  that these stepsizes have been discussed extensively in the related literature, including \cite{AlberIusemSolodov1998, GoffinKiwiel1999, nedic_bertsekas2001rate, nedic_bertsekas2001, NedicBertsekas2010, xmwang2018}, where many of our results were inspired.  Let us describe the results   in the present and their  relationship with the literature on the subject.  With respect to the exogenous stepsize we establish convergence results without any compactness assumption, existence of a solution, and  the iteration-complexity bound, which  are similar to the well known bound presented in \cite{AlberIusemSolodov1998, nedic_bertsekas2001}.   In particular, for $C = \mathbb{R}^n$, the convergence results merge into the ones presented in \cite{CorreaLemarecha1993} and the iteration-complexity bound  into \cite[Theorem~3.2.2]{Nesterov2004}.  The asymptotic convergence result and the iteration-complexity bound obtained using Polyak's stepsizes are similar to the corespondent ones in  \cite{ nedic_bertsekas2001, Nesterov2014, Polyak1969} and  \cite{Nesterov2014}, respectively.  Regarding  to the dynamic stepsize, we establish global convergence in objective values as address, for example, in  \cite{GoffinKiwiel1999, nedic_bertsekas2001}. In \cite[Proposition 2.15]{nedic_bertsekas2001rate}, the authors presented the rate of convergence for another variant of subgradient method, known as incremental subgradient algorithms. This study allowed us to estimate an iteraction-complexity bound for the dynamic stepsize.

The organization of the paper is as follows. In Section \ref{sec:int.1}, we present some notation and basic results used in our presentation. In Section \ref{Sec:SubInexProj} we describe the Subgradient-InexP method with different choices for the   stepsize. The main results of the present paper, including the converge theorems and iteration-complexity, are presented in Section \ref{Sec:aca}.   Some numerical experiments are provided in Section \ref{Sec:NumExp}.   We conclude the paper with some remarks in Section \ref{Sec:Conclusions}.
\section{Notation and definitions} \label{sec:int.1}

In this section, we present  some notations, definitions, and results used throughout the paper.  We  are interested in 
\begin{equation} \label{eq:OptP}
	\min \{ f(x) :~  x\in C\},
\end{equation}
where $C$ is a closed and convex subset of $\mathbb{R}^n$, $f:\mathbb{R}^n \to \mathbb{R}$ is a convex function. We denote by 
\begin{equation} \label{eq:ValueOpt}
f^*:= \inf_{x\in C} f(x),
\end{equation}
its infimal value (possibly $-\infty$) and by $\Omega^*$ its solution set (possibly  $\Omega^*= \varnothing$). The next   concept will be useful in the analysis of the sequence generated by the subgradient method to solve \eqref{eq:OptP}.
\begin{definition} \label{def:QuasiFejer}
	A sequence $\{y_k\}\subset \mathbb{R}^n$ is said to be  quasi-Fej\'er convergent to a nonempty set $W\subset \mathbb{R}^n$ if, for every $w\in W$, there exists a sequence $\{\delta_k\}\subset\mathbb{R}$ such that $\delta_k\geq 0$, $\sum_{k=1}^{\infty}\delta_k<+\infty$, and $$\|y_{k+1}-w\|^2\leq \|y_k-w\|^2+\delta_k, \qquad \forall~k=0, 1, \ldots.$$
When, $\delta_k= 0$, for all $k=0, 1, \ldots.$,  $\{y_k\}$  is called  Fej\'er convergent  to a set $W$. 
\end{definition}
The main property of the  quasi-Fej\'er convergent sequence is stated  in the next result, and its proof can be found in \cite{burachik1995full}.
\begin{theorem}\label{teo.qf}
	Let $\{y_k\}$ be a sequence in  $\mathbb{R}^n$. If $\{y_k\}$ is quasi-Fej\'er convergent to a nomempty set $W\subset  \mathbb{R}^n$, then $\{y_k\}$ is bounded. If furthermore, a cluster point $y$ of $\{y_k\}$ belongs to $W$, then $\lim_{k\rightarrow\infty}y_k=y$.
\end{theorem}

To describe the method  for solving the problem~\eqref{eq:OptP}  we need  to define, for each $\epsilon \geq 0$,  the $\epsilon$-subdifferential  $\partial_{\epsilon} f(x)$ of a convex function $f$  at $x\in {\mathbb R}^n$,
\begin{equation} \label{eq:e-subdif}
	\partial_{\epsilon} f(x):=\{ s\in {\mathbb R}^n:~f(y)\geq f(x)+\langle s, y-x\rangle -\epsilon, ~\forall y\in {\mathbb R}^n\}.
\end{equation}
We end this section by presenting  important properties  of the set $\epsilon$-subdifferential of a convex function, which   proofs  follow by combining  \cite[Proposition 4.3.1(a)]{Bertsekas2003} and  \cite[Proposition 4.1.1, Proposition 4.1.2]{UrrutyLemarechal1993_II}.
\begin{proposition} \label{pr:CompE-subdif}
	Let $f:\mathbb{R}^n \to \mathbb{R}$ be a convex function and $\epsilon \geq 0$. The set $\partial_{\epsilon} f(x)$ is nonempty, convex,  and compact.  Moreover, if $B\subset  \mathbb{R}^n$ is a bounded set, then there exists a  real number $L>0$ such that $\|s\|<L$, for all $s\in \cup_{x\in B} \partial_{\epsilon} f(x)$. In addition, if   $\{\epsilon_k\}$ is a bounded  sequence of nonnegative real numbers, the sequence $\{x_k\}$ converges to $x \in \mathbb{R}^n$,  and $s_k \in \partial_{\epsilon_k} f(x_k)$  for all $k$, then the sequence $\{s_k\}$ is bounded.
\end{proposition}
\section{Subgradient-InexP method} \label{Sec:SubInexProj}
Next, we present the subgradient method with a feasible  inexact  projections, which will be called  {\it Subgradient-InexP method}.  We begin  by presenting the concept of relative feasible inexact projection, which is a variation of those presented in   \cite{OrizonFabianaGilson2018, VillaSalzBaldassarre2013}.
\begin{definition} \label{def:InexactProj}
Let $C\subset {\mathbb R}^n$ be a closed convex set and  $\varphi_{\gamma, \theta, \lambda}: {\mathbb R}^n\times {\mathbb R}^n\times {\mathbb R}^n \to {\mathbb R}_{+}$ be a relative error tolerance function such that
	\begin{equation} \label{eq:fphi} 
		\varphi_{\gamma, \theta, \lambda}(u, v, w)\leq \gamma \|v-u\|^2 + \theta \|w-v\|^2 +   \lambda \|w-u\|^2, \qquad \forall~ u, v, w \in \mathbb{R}^n,
	\end{equation}
	where $ \gamma, \theta, \lambda \geq 0$ are given forcing parameters. 
 The {\it feasible inexact projection mapping} relative to $u \in C$ with relative error tolerance function  $\varphi_{\gamma, \theta, \lambda}$, denoted by ${\cal P}_C(\varphi_{\gamma, \theta, \lambda},u, \cdot): {\mathbb R}^n \rightrightarrows C$ is the set-valued mapping defined as follows
	\begin{equation} \label{eq:ProjI} 
		{\cal P}_C(\varphi_{\gamma, \theta, \lambda}, u, v) := \left\{w\in C:~\left\langle v-w, z-w \right\rangle \leq \varphi_{\gamma, \theta, \lambda}(u, v, w), \quad \forall~ z \in C \right\}.
	\end{equation}
Each point $w\in {\cal P}_C(\varphi_{\gamma, \theta, \lambda}, u, v)$ is called a {\it feasible inexact projection of $v$ onto $C$ relative to $u$ and with relative error tolerance function $\varphi_{\gamma, \theta, \lambda}$}.
\end{definition}
In the following, we present  some remarks about the definition of the feasible inexact projection mapping onto the convex set $C$.
\begin{remark}\label{rem: welldef}
Let   $C\subset {\mathbb R}^n$,    $u\in C$   and $\varphi_{\gamma, \theta, \lambda}$ be as in Definition~\ref{def:InexactProj}. Therefore,    for all $v\in {\mathbb R}^n$,   it follows from \eqref{eq:ProjI}  that  ${\cal P}_C(0, u, v)$ is  the exact projection of $u$ onto $C$; see \cite[Proposition~2.1.3, p. 201]{Bertsekas1999}. Moreover,  ${\cal P}_C(0, u, v) \in {\cal P}_C(\varphi_{\gamma, \theta, \lambda}, u, v)$ concluding  that ${\cal P}_C(\varphi_{\gamma, \theta, \lambda}, u, v)\neq \varnothing$, for all $u\in C$ and $v\in {\mathbb R}^n$. Consequently,  the set-valued mapping ${\cal P}_C(\varphi_{\gamma, \theta, \lambda}, u, \cdot) $ is well-defined.
\end{remark}
Next lemma is a variation of \cite[Lemma 6]{Reiner_Orizon_Leandro2019}. It will play an important role in the remainder of this paper. 

\begin{lemma} \label{pr:cond}
	Let $v \in {\mathbb R}^n$, $u \in C$, $\gamma, \theta, \lambda \geq 0$ and $w\in {\cal P}_C(\varphi_{\gamma, \theta, \lambda}, u, v)$. Then, there holds	
	$$\|w-x\|^2 \leq \|v-x\|^2 + \frac{2\gamma+2\lambda}{1-2\lambda}\|v-u\|^2, \qquad \forall ~x \in C,$$
	for all $ \lambda, \theta \in [0,  1/2)$.
\end{lemma}

\begin{proof}
		Let $x \in C$. First note that $\|w-x\|^2 = \|v-x\|^2 - \|w-v\|^2 + 2 \langle v-w, x-w \rangle$. Since  $w \in {\cal P}_C(\varphi_{\gamma, \theta, \lambda}, u, v)$ and  $0\leq \theta < 1/2$, combining the last  equality with \eqref{eq:fphi} and \eqref{eq:ProjI} we obtain
	\begin{align} 
		\|w-x\|^2 &\leq \|v-x\|^2 - (1-2\theta)\|v-w\|^2 + 2\gamma \|v-u\|^2 + 2\lambda \|w-u\|^2\notag\\
		              &\leq \|v-x\|^2  + 2\gamma \|v-u\|^2 + 2\lambda \|w-u\|^2 \label{eq:fg}.
	\end{align}
	On the other hand, we also have
	\begin{align*}
		\|w-u\|^2 &= \|v-u\|^2 + \|w-v\|^2 + 2 \langle v-w,u-v \rangle \\
							&= \|v-u\|^2 + \|w-v\|^2 + 2 \langle v-w,u-w \rangle - 2 \|w-v\|^2 \\
							&= \|v-u\|^2 - \|w-v\|^2 + 2 \langle v-w,u-w \rangle.
	\end{align*}
	Thus, due to $w\in {\cal P}_C(\varphi_{\gamma, \theta, \lambda}, u, v)$ and $u \in C$, using \eqref{eq:fphi}, \eqref{eq:ProjI}, $0\leq \theta < 1/2$ and $0 \leq \lambda <1/2$,  we have
	$$\|w-u\|^2 \leq \frac{1+2\gamma}{1-2\lambda}\|v-u\|^2 - \frac{1-2\theta}{1-2\lambda} \|w-v\|^2\leq \frac{1+2\gamma}{1-2\lambda}\|v-u\|^2.$$
	Therefore, combining the last inequality with \eqref{eq:fg},  we obtain the desired inequality.
\end{proof}

The conceptual  subgradient method with feasible inexact projections for solving the Problem~\eqref{eq:OptP} is formally defined as follows:\\

\begin{algorithm} [H]
\begin{description}
	\item[ Step 0.] Let $\{\epsilon_k\}$, $\{\theta_k\}$,  and $\{\lambda_k\}$ be sequences of nonnegative real numbers. Let $x_0\in C$ and set $k=0$.
	\item[ Step 1.] If $0\in \partial f(x_k)$, then {\bf stop}. Otherwise, choose a non-null  element $s_k \in \partial_{\epsilon_k} f(x_k)$, compute a  stepsize $t_k>0$,  (to be specified later),    and take the next iterate as any point such that 
$$
		x_{k+1} \in {\cal P}_C\left(\varphi_{ \gamma_k, \theta_k, \lambda_k}, x_k, x_k-t_ks_k\right).
$$
	\item[ Step 2.]  Set $k\gets k+1$, and go to \textbf{Step~1}.
\end{description}
\caption{Subgradient-InexP method}
\label{Alg:INP}
\end{algorithm}
\bigskip
\noindent
Let us  describe the main features of the subgradient-InexP method. Firstly,  we check if the current iterate $x_k$ is a solution of Problem \eqref{eq:OptP}.  If  $x_k$ is not a a solution, then we choose a non-null  element $s_k \in \partial_{\epsilon_k} f(x_k)$, compute a stepsize  $t_k>0$,  and take the next iterate $x_{k+1}\in C$ as any feasible inexact projection of $x_k-t_ks_k$ onto $C$ relative to $x_{k}$ with error tolerance given by  $\varphi_{ \gamma_k, \theta_k, \lambda_k}(x_k, x_k-t_ks_k,x_{k+1})  $, i.e., $x_{k+1} \in {\cal P}_C\left(\varphi_{ \gamma_k, \theta_k, \lambda_k}, x_k, x_k-t_ks_k\right)$.  We remark that if $\varphi_{ \gamma_k, \theta_k, \lambda_k}\equiv~0$, then ${\cal P}_C\left(0, x_k, x_k-t_ks_k\right)$ is the exact projection of $x_k-t_ks_k$ onto $C$,  and Algorithm~\ref{Alg:INP} amounts to the projected subgradient method  studied in \cite{AlberIusemSolodov1998}. Among the several possible choices that appeared in literature on subject,  see for example \cite{Bertsekas1999, nedic_bertsekas2001, Shor1985}, we studied three well known strategies, beginning with exogenous stepsize.
\begin{Rule}[Exogenous stepsize]\label{Exogenous.Step} 
	Let $ \mu\geq 0$. Take exogenous sequences $\{\alpha_k\}$ and $\{\epsilon_k\}$ of nonnegative real numbers satisfying the following conditions: the sequence $\{\epsilon_k\}$ is nonincreasing and 
	\begin{equation}\label{eq:ExogSeq}
		\sum_{k=0}^{\infty}\alpha_k=+\infty, \qquad \qquad  \sum_{k=0}^{\infty}\alpha_k^2<+\infty, \qquad \qquad   \epsilon_k\leq \mu \alpha_ k, \qquad \qquad  ~k=0, 1, \ldots.  
	\end{equation} 
Given $s_k \in \partial_{\epsilon_k} f(x_k)$, define the stepsize $t_k $ as  the following   nonnegative real number 
	\begin{equation} \label{eq:StepSize1}
		t_k:=\frac{\alpha_k}{\eta_k},   \qquad \qquad    \eta_k:= \max\left\{1, \| s_k\|\right\},  \qquad  \qquad  ~k=0, 1, \ldots.
	\end{equation}
\end{Rule}

\noindent
The stepsize in Rule~\ref{Exogenous.Step} is one the most popular. It have been used in several paper for analyzing subgradient method; see for example, \cite{AlberIusemSolodov1998,CorreaLemarecha1993,nedic_bertsekas2001rate,nedic_bertsekas2001,xmwang2018}.  

{\it From now on we assume that there exist  $0\leq \bar{\theta} < 1/2$ and $0\leq \bar{\lambda} < 1/2$, such that  $\{\theta_k\}\subset [0, {\bar \theta})$, $\{\gamma_k\}\subset [0, {\bar \gamma})$ and $\{\lambda_k\}\subset [0, {\bar \lambda})$. For future references define}

\begin{equation} \label{eq:nu}
	\nu := \frac{1+2{\bar \gamma}}{1-2{\bar \lambda}}> 0.
\end{equation}
To define the next stepsize, we need to known the optimum value $f^*$ given in  \eqref{eq:ValueOpt}. In \cite[p.142]{PolyakBook}, is present some examples of problems for which   the optimum value are known. The statement of the Polyak's stepsize is as follows.

\begin{Rule}[Polyak's stepsize]\label{Poliak.Step} 
	Assume that $\Omega^*\neq\varnothing$ and the optimal value   $f^*>-\infty$ is known. Let $\mu\geq 0$, $ \underline{\beta} >0$,  ${\bar \beta} >0$ and take exogenous sequences $\{\beta_k\}$ and $\{\epsilon_k\}$ of real numbers satisfying the following conditions: the sequence $\{\epsilon_k\}$ is nonincreasing and 

	\begin{equation}\label{beta}
		0< \underline{\beta} \leq \beta_k \leq \bar{\beta} < \frac{1}{2\mu +\nu}, \qquad\qquad  0<\epsilon_k\leq \mu \beta_ k[f(x_k)-f^*], \quad \qquad  ~k=0, 1, \ldots.
	\end{equation}
Given $s_k \in \partial_{\epsilon_k} f(x_k)$, $s_k\neq 0$, define the stepsize $t_k $ as  the following nonnegative real number 
	\begin{equation}\label{StepsizePolyak}
		t_k=\beta_k\frac{f(x_k)-f^*}{\left\|s_k\right\|^2},   \qquad \qquad  ~k=0, 1, \ldots.
	\end{equation}
\end{Rule}

\noindent
The stepsize in Rule~\ref{Poliak.Step} was introduced in \cite{Polyak1969} and has been used in several papers, including the ones  \cite{nedic_bertsekas2001rate,nedic_bertsekas2001, xmwang2018}. In general, in practical problems the optimum value \eqref{eq:ValueOpt} is not known. In this case, we may modify the stepsize \eqref{StepsizePolyak} by replacing the optimum value \eqref{eq:ValueOpt} with a suitable estimate in each iteration. This leads to the dynamic stepsize rule as follows.

\begin{Rule}[Dynamic stepsize]\label{Dynamic.Step} 
Let $\mu\geq 0$,  $ \underline{\beta} >0$,  ${\bar\beta}>0$,  and take exogenous sequences $\{\beta_k\}$ and $\{\epsilon_k\}$ of real numbers satisfying the following conditions: the sequence $\{\epsilon_k\}$ is nonincreasing and
	\begin{equation}\label{betaDinamic}
		0< \underline{\beta} \leq \beta_k \leq \bar{\beta} < \frac{2}{2\mu+\nu}, \qquad \qquad 0<\epsilon_k\leq \mu \beta_ k[f(x_k)-f_{k}^{lev}], \quad \qquad  ~k=0, 1, \ldots, 
	\end{equation}
where $f_{k}^{lev}$ will be specified later (see Section~\ref{Sec:Analyisdynamic}).  Given $s_k \in \partial_{\epsilon_k} f(x_k)$ such that  $s_k\neq 0$, define the stepsize $t_k $ as  the following nonnegative real number 
	\begin{equation}\label{eq.dynamicstep}
		t_k= \frac{\tilde{t}_k}{\left\|s_k\right\|}, \qquad\qquad  \tilde{t}_k = \beta_k\frac{f(x_k)-f_{k}^{lev}}{\left\|s_k\right\|}, \qquad\qquad  ~k=0, 1, \ldots.
	\end{equation}
\end{Rule}

\noindent
The dynamic stepsize  in Rule \ref{Dynamic.Step} is based on the ideas of \cite{brannlund1993}; see also \cite{GoffinKiwiel1999}. This rule has been used in several papers, see for example \cite{nedic_bertsekas2001,NedicBertsekas2010, xmwang2018}.  

{\it From now on we assume that the sequence $\{x_k\}$ is generated by Algorithm~\ref{Alg:INP}, with one  of the three above strategies for choosing the stepsize, is  infinite.}
\section{Analysis of the subgradient-InexP method} \label{Sec:aca}
In the following, we state and prove our first result to analyze the sequence $\{x_k\}$ generated by Algorithm~\ref{Alg:INP}. The  obtained inequality in next lemma is its counterpart for unconstrained optimization provided in \cite[Lemma 1.1)]{CorreaLemarecha1993}. As we shall see, this inequality will be the main tool in our asymptotic convergence analysis, as well as in the iteration-complexity analysis. 
\begin{lemma} \label{Le:FejerConv}
	Let $\nu>0$  be as defined \eqref{eq:nu}. For all $x \in C$, the following inequality holds
	\begin{equation} \label{eq:MainIneq}
		\|x_{k+1}-x\|^2 \leq \|x_k-x\|^2 + \nu t_k^2\|s_k\|^2 - 2 t_k \left[f(x_k) - f(x) -\epsilon_k\right], \qquad k=0, 1, \ldots. 
	\end{equation}
\end{lemma}

\begin{proof}
	Let $x \in C$. To simply the notations we set $z_k:= x_k-t_ks_k$. Due to $x_{k+1} \in {\cal P}_C\left(\varphi_{\gamma_k, \theta_k, \lambda_k}, x_k, z_k\right)$ and $x_k \in C$, we apply Lemma~\ref{pr:cond} with $w=x_{k+1}$, $v=z_k$,  $u=x_k$,   and $\varphi_{ \gamma, \theta, \lambda}=\varphi_{ \gamma_k, \theta_k, \lambda_k}$ to conclude 
	\begin{equation} \label{eq:mip}
		\|x_{k+1}-x\|^2 \leq \|z_k-x\|^2+ \frac{2\gamma_k+2\lambda_k}{1-2\lambda_k}t_k ^2\|s_k\| ^2.
	\end{equation}
	On the other hand, due to  $z_k= x_k-t_ks_k$,  after some algebraic manipulations,  we obtain 
	$$\|z_k-x\|^2 = \|x_k-x\|^2 + t_k ^2\|s_k\| ^2 + 2t_k \langle s_k, x - x_k \rangle.$$
	Since $s_k \in \partial_{\epsilon_k} f(x_k)$, the definition \eqref{eq:e-subdif} implies that $\langle s_{k}, z-x_k\rangle\leq  f(z)- f(x_k)+\epsilon_k$. Thus, 
	$$\|z_k-x\|^2 \leq \|x_k-x\|^2 + t_k ^2\|s_k\| ^2 +  2t_k\left[f(x) - f(x_k) + \epsilon_k\right].$$
	Therefore, combining last inequality with \eqref{eq:mip} we conclude that 	
	\begin{align*} 
	\|x_{k+1}-x\|^2 &\leq \|x_k-x\|^2 + t_k ^2\|s_k\|^2 + 2t_k\left[f(x) - f(x_k) + \epsilon_k\right] + \frac{2\gamma_k+2\lambda_k}{1-2\lambda_k}t_k^2 \|s_k\|^2\\
	                       &= \|x_k-x\|^2 + \frac{1+2\gamma_k}{1-2\lambda_k}t_k^2 \|s_k\|^2- 2 t_k \left[f(x_k) - f(x) -\epsilon_k\right].
	\end{align*}
 Considering that $0 \leq \lambda_k < \bar{\lambda}<1/2$ and $0 \leq \gamma_k < \bar{\gamma}$, and using \eqref{eq:nu}, we obtain \eqref{eq:MainIneq}.
\end{proof}

\subsection{Analysis of the subgradient-InexP method with exogenous stepsize} \label{Sec:AnalyisExog}
In this section we will analyze the subgradient-InexP method with stepsizes satisfying Rule~\ref{Exogenous.Step}. For that, {\it throughout this section we assume also that $\{x_k\}$ is a sequence generated by Algorithm~\ref{Alg:INP} with the stepsize given by  Rule~\ref{Exogenous.Step} and, define
\begin{equation} \label{eq:rho}
	\rho:= \nu + 2\mu > 0.
\end{equation}
}
First of all, note that under the above  assumptions,  Lemma~\ref{Le:FejerConv} becomes as follows.
\begin{lemma} \label{Le:FejerConvExog}
	Let $\rho>0$  be as in  \eqref{eq:rho}. For all $x\in C$, the following inequality holds
		$$\|x_{k+1}-x\|^2\leq \|x_k-x\|^2 + \rho \alpha_k^2 - 2 \frac{\alpha_k}{\eta_k} \left[f(x_k)-f(x)\right], \qquad k=0, 1, \ldots.$$
\end{lemma}
\begin{proof}
	The definition of $t_k$ in \eqref{eq:StepSize1} implies $t_k\leq \alpha_k$, which combined with the last inequality in  \eqref{eq:ExogSeq} yields $2t_k\epsilon_k\leq 2\mu \alpha_k^2$. Moreover, \eqref{eq:StepSize1} also implies that $t_k^2\|s_k\|^2\leq \alpha_k^2$. Therefore, using \eqref{eq:rho}, the desired inequality follows directly from  \eqref{eq:MainIneq}.
\end{proof}
To proceed with the analysis of Algorithm~\ref{Alg:INP},  we also need the following auxiliary set 
\begin{equation} \label{eq:DefOmega}
	\Omega:=\left\{x\in C:~f(x)\leq \inf_kf(x_k)\right\}.
\end{equation}
It is worth mentioning that, in principle,  set $\Omega$ can be empty and,  in such case,  $f^*=-\infty$. In the next lemma we analyze the behavior of the sequence $\{x_k\}$ under the hypothesis that $\Omega \neq \varnothing$. 
\begin{lemma} \label{Le:BoundExog}
  If $\Omega \neq \varnothing$, then $\{x_k\}$ is quasi-F\'ejer convergent to $\Omega$. Consequently, $\{x_k\}$ is bounded. 
\end{lemma}
\begin{proof}
	Since $\Omega \neq \varnothing$, take  $x\in \Omega$. Thus, by using the definition of $\Omega$ in \eqref{eq:DefOmega} and Lemma~\ref{Le:FejerConvExog}, we conclude that $\|x_{k+1}-x\|^2\leq \|x_k-x\|^2 + \rho \alpha_k^2$, for all $k=0, 1, \ldots$. Hence, using the first inequality in \eqref{eq:ExogSeq}, the first statement of the lemma follows from Definition~\ref{def:QuasiFejer}. The second statement of the lemma follows from the first part of Theorem~\ref{teo.qf}.
\end{proof}
Now, we are ready to prove the main result of this section, which refers to the asymptotic convergence of $\{x_k\}$. We remark that in the first part of the next theorem we do not assume neither  $\Omega^* \neq \varnothing$ nor that  $f^*$ is  finite.

\begin{theorem}\label{teo.Main}
	The following equality holds

	\begin{equation} \label{eq:linfs}
		\liminf_k f(x_k)=f^*.
	\end{equation}
	In addition, if $\Omega^* \neq \varnothing$ then the sequence $\{x_k\}$ converges to a point $x_*\in\Omega^*$.
\end{theorem}

\begin{proof}
	Assume by contradiction that $\liminf_k f(x_k) > f^*$. In this case, we have $\Omega \neq \varnothing$. Consequently by Lemma \ref{Le:BoundExog}, we conclude that $\{x_k\}$ is bounded. Letting $x\in \Omega$, there exist $\tau>0$ and $k_0\in\mathbb{N}$ such that $f(x)<f(x_k)-\tau,$ for all $k\geq k_0$. Hence, using Lemma~\ref{Le:FejerConvExog}, we have
	\begin{equation}\label{eq:bXk}
		\|x_{k+1}-x\|^2 \leq \|x_{k}-x\|^2 + \rho \alpha_k^2- 2 \frac{\alpha_ k}{\eta_k} \tau, \qquad  k=k_0, k_0+1, \dots.
	\end{equation}
	On the other hand, it follows  from  \eqref{eq:ExogSeq}  that the sequence $\left\lbrace \epsilon_k \right\rbrace$ is bounded.  Thus,  considering that $\{x_k\}$ is bounded,  Proposition~\ref{pr:CompE-subdif} implies that $\{s_k\}$ is also bounded. Let $c > 0$  be such that $\|s_k\| \leq c$, for all $k \geq  0$. Hence,  using  second equality in  \eqref{eq:StepSize1}, we have  $\eta_k = \max\left\{1, \| s_k\|\right\} \leq \max\left\{1, c \right\} =: \varGamma$. Thus, letting $\ell \in \mathbb{N}$ and using \eqref{eq:bXk}, we conclude that
	$$\frac{2\tau}{\varGamma}\sum_{j=k_0}^{\ell+k_0}\alpha_j \leq \|x_{k_0}-x\|^2 - \|x_{k_0+ \ell+1}-x\| + \rho \sum_{j=k_0}^{\ell+k_0}\alpha^2_j \leq \|x_{k_0}-x\|^2 + \rho \sum_{j=k_0}^{\ell+k_0}\alpha^2_j.$$
	Since the last inequality holds for all $\ell \in \mathbb{N}$ then, by  using the first two conditions on $\{\alpha_k\}$ in \eqref{eq:ExogSeq},  we have a contraction. Therefore, \eqref{eq:linfs} holds. For proving the last statement, let us assume that $\Omega^*\neq\varnothing$. In this case, we also have $\Omega\neq\varnothing$ and, from Lemma~\ref{Le:BoundExog}, the sequence $\{x_k\}$ is bounded and   quasi-F\'ejer convergent to $\Omega$. The equality \eqref{eq:linfs} implies that $\{f(x_k)\}$ has a decreasing monotonous subsequence $\{f(x_{k_j})\}$ such that $\lim_{j\rightarrow \infty}f(x_{k_j})= f^*.$ Without lose of generality, we can assume  that  $\{f(x_k)\}$ is decreasing, is monotonous, and converges to $f^*$. Being bounded, the sequence $\{x_k\}$ has  a convergent subsequence $\{x_{k_\ell}\}$. Let us say that $\lim_{\ell\rightarrow\infty}x_{k_\ell}=x_*,$ which by the continuity of $f$ implies $f(x_*)=\lim_{\ell\rightarrow\infty}f(x_{k_\ell})=f^*,$ and then $x_*\in\Omega^*$. Hence, $\{x_k\}$ has an cluster point $x_*\in\Omega$, and due to $\{x_k\}$ be quasi-F\'ejer convergent to $\Omega$, Theorem~\ref{teo.qf} implies that $\{x_k\}$ converges to $x_*$.
\end{proof}
Next theorem presents an iteration-complexity bound; similar bound can be found in   \cite[Theorem~3.2.2]{Nesterov2004}.
\begin{theorem}\label{teo:complrule1}
	Assume that the sequence $\{x_k\}$ converges to a point $x_*\in\Omega^*$. Then, for every $N \in \mathbb{N}$, the following inequality holds
	$$\min \{f(x_k) - f^{*}:~ \, k = 0, 1,\ldots, N \} \leq \varGamma \frac{\|x_0 - x_{*}\|^2 + \rho\sum_{k=0}^{N}\alpha_k^{2}}{2\sum_{k=0}^{N}\alpha_k}.$$
\end{theorem}

\begin{proof}
	Since $\left\lbrace \epsilon_k \right\rbrace $ and $\left\lbrace x_k\right\rbrace $ are bounded sequences, then using Proposition \ref{pr:CompE-subdif}, it follows that $\left\lbrace s_k\right\rbrace$ is  also bounded, i.e. there exists $c > 0$ such that $\|s_k\| \leq c$, for all $k \geq 0$. Therefore, using the definition of $\eta_k$ in \eqref{eq:StepSize1},  we have $\eta_k = \max\left\{1, \| s_k\|\right\} \leq \max\left\{1, c \right\} =: \varGamma$. Now, applying Lemma~\ref{Le:FejerConvExog} with $x = x_*$ and due to  $f^* = f(x_*)$, we obtain
	$$\frac{2 \alpha_k}{\varGamma} [f(x_k)-f^*]  \leq \|x_k-x_*\|^2 -\|x_{k+1}-x_*\|^2+ \rho\alpha_k^2 , \qquad k = 0,1, \ldots.$$
	Hence, performing the sum of the above inequality for $k = 0, 1, \ldots, N,$ we have
	$$\frac{2}{\varGamma} \sum_{k=0}^{N} \alpha_k [f(x_k)-f^*] \leq \|x_0-x_*\|^2-\|x_{N+1}-x_*\|^2+\rho\sum_{k=0}^{N}\alpha_k^2.$$
	Therefore,
	$$\frac{2}{\varGamma} \, \min \left\{f(x_k) - f^{*}: ~ \, k = 0, 1,\ldots, N \right\} \sum_{k=0}^{N}\alpha_k \leq \|x_0 - x_{*}\|^2 +\rho\sum_{k=0}^{N}\alpha_k^{2},$$
	which is equivalent to the desired inequality.
\end{proof}
\subsection{Analysis of the subgradient-InexP method with Polyak's stepsize rule} \label{Sec:AnalysisPolyak}
In this section we will analyze the  subgradient-InexP method with Polyak's step sizes.  {\it Throughout this section,  we assume also that $\Omega^* \neq \varnothing$ and $\{x_k\}$ is a sequence generated by Algorithm~\ref{Alg:INP} with the stepsize given by Rule~\ref{Poliak.Step}}.

\begin{lemma} \label{Le:FejerConvPolyak}
	Let $x\in \Omega^*$. Then, the following inequality holds
	\begin{equation} \label{eq:MainIneqPolyak}
		\|x_{k+1}-x\|^2\leq \|x_k-x\|^2 - \underline{\beta} \frac{\left[f(x_k)-f^*\right]^2}{\|s_k\|^2}, \qquad\qquad k=0, 1, \ldots.
	\end{equation}
\end{lemma}

\begin{proof}
Considering  that  $x\in \Omega^*$ we have $f^*= f(x)$.  The combination of \eqref{beta} with \eqref{StepsizePolyak} implies $2t_k\epsilon_k\leq 2\mu \beta_k^2 [f(x_k)-f^*]^2/\|s_k\|^2$.  Moreover, \eqref{StepsizePolyak} also implies that  $t_k^2\|s_k\|^2 = \beta_k^2[f(x_k)-f^*]^2/\|s_k\|^2$.  Thus, we conclude form  \eqref{eq:MainIneq}  that
	\begin{equation} \label{eq:mdpss}
	\|x_{k+1}-x\|^2\leq \|x_k-x\|^2 -\left(2-\nu\beta_k-2\mu \beta_k\right)\beta_k\frac{\left[f(x_k)-f^*\right]^2}{\|s_k\|^2}, \qquad k=0, 1, \ldots. 
	\end{equation}
On the other hand, \eqref{beta}  gives us   $\beta_k<1/(2\mu+\nu)$, which is equivalent to  $2-\nu\beta_k-2\mu \beta_k>1$. Therefore, since   \eqref{beta}  also gives $\underline{\beta}\leq \beta_k$, we conclude that \eqref{eq:mdpss} implies \eqref{eq:MainIneqPolyak}.  
\end{proof}
In the following theorem we present our main result about the asymptotic   convergence of $\{x_k\}$. It has as correspondent  result in  \cite[Theorem 1]{Polyak1969}; see also \cite{nedic_bertsekas2001}.

\begin{theorem}\label{teo.MainConvPolyak}
 The sequence $\{x_k\}$ converges to a point $x_*\in\Omega^*$.
\end{theorem}

\begin{proof} 
	Let $x\in \Omega^*$. Then, Lemma~\ref{Le:FejerConvPolyak} implies $\|x_{k+1}-x\|^2\leq \|x_k-x\|^2 $, for all $k=0, 1, \ldots$. Thus, $\{x_k\}$ is Fej\'er convergent to  $\Omega^*$. Since $\Omega^* \neq \varnothing$, Theorem~\ref{teo.qf} implies that $\{x_k\}$ is bounded. By using Proposition~\ref{pr:CompE-subdif}, we conclude that there exists $c>0$ such that $\left\|s_k\right\| \leq c$, for $k=0,1,\ldots$. Then, from \eqref{eq:MainIneqPolyak},  after some algebra, we have 
$$\left[f(x_k)-f^*\right]^2 \leq \frac{c^2}{\underline{\beta}}\left(\|x_k-x\|^2 - \|x_{k+1}-x\|^2\right), \qquad \qquad k=0, 1, \ldots.$$
Thus, performing the sum of the this inequality for $j=0, 1, \ldots, \ell$, we obtain 
$$\sum_{j=0}^{\ell}\left[f(x_j)-f^*\right]^2 \leq \frac{c^2}{\underline{\beta}}\left(\|x_{0}-x\|^2 - \|x_{ \ell+1}-x\|^2\right)\leq \frac{c^2}{\underline{\beta}} \|x_{0}-x\|^2.$$
Considering that this inequality  holds for all $\ell\in \mathbb{N}$, we conclude that $\lim_{k\to +\infty}f(x_k)=f^*$. Let $x_*$ be a cluster point of $\{x_k\}$ and $\{x_{k_j}\}$ a subsequence of $\{x_k\}$ such that $\lim_{j\to +\infty}x_{k_j}=x_*$. Since $f$ is  continuous, we have $f(x_*)= \lim_{j\to +\infty}f(x_{k_j})=f^*$. Therefore, $x_*\in\Omega^*$. Since $\{x_k\}$ is quasi-Fej\'er convergent to a set $\Omega^*$, it follows from Theorem~\ref{teo.qf} that $\{x_k\}$ converges $x_*$.
\end{proof}
The next result presents an iteration-complexity bound, which is a version of \cite[Theorem~1]{Nesterov2014}.
\begin{theorem}
	Assume that the sequence $\{x_k\}$ converges to a point $x_*\in\Omega^*$. Then, for every $N \in \mathbb{N}$, the following inequality holds
	$$\min \{f(x_k) - f^{*}: ~ \, k = 0, 1,\ldots, N \} \leq \frac{c}{\sqrt{\underline{\beta} (N+1)}} \|x_0 - x_{*}\|,$$
	where $c\geq \max\{\|s_k\|:~ k=0, 1,\ldots \}$.
\end{theorem}

\begin{proof}
Applying Lemma~\ref{Le:FejerConvPolyak} with $x = x_*$ , where  $f^* = f(x_*)$, we obtain
	$$
	\underline{\beta} \frac{\left[f(x_k)-f^*\right]^2}{\|s_k\|^2} \leq \|x_k-x_*\|^2 - \|x_{k+1}-x_*\|^2, \qquad\qquad k=0, 1, \ldots.
	$$
Performing the sum of the above inequality for $k = 0, 1, \ldots, N,$ we conclude that
	$$
	\sum_{k=0}^{N} \frac{\left[f(x_k)-f^*\right]^2}{\|s_k\|^2} \leq \frac{1}{\underline{\beta}} \|x_0-x_*\|^2.
	$$
Since  $\{x_k\}$ is bounded, by using Proposition~\ref{pr:CompE-subdif}, we conclude that there exists $c>0$ such that $\left\|s_k\right\| \leq c$, for $k=0,1,\ldots$. Thus, we have
	$$\sum_{k=0}^{N} \left[f(x_k)-f^*\right]^2 \leq \frac{c^2}{\underline{\beta}} \|x_0-x_*\|^2.$$
Therefore, 
$$(N+1) \min \{\left[f(x_k)-f^*\right]^2:~ \, k = 0, 1,\ldots, N \} \leq \frac{c^2}{\underline{\beta}} \|x_0-x_*\|^2,
$$
which is equivalent to the desired inequality.
\end{proof}

\subsection{Analysis of the subgradient-InexP method with dynamic stepsize} \label{Sec:Analyisdynamic}
Next  we consider the Subgradient-InexP method employing the dynamic stepsize  Rule~\ref{Dynamic.Step},  which guarantees that $\{f_{k}^{lev}\} $  converges to the optimum value $f^*$. In the following  we present   formally the   algorithm  which compute $f_{k}^{lev}$. This scheme was introduced in \cite{brannlund1993}; see also \cite{GoffinKiwiel1999}.

\begin{algorithm} 
\begin{description}
	\item[Step 0.] Select $x_0\in C, \delta_0 > 0$,  and $R > 0$. Set $k=0, \sigma_0 = 0, f_{-1}^{rec} = \infty, \ell=0, k(\ell) = 0$.
	\item[Step 1.] If $f(x_k) < f_{k-1}^{rec},$ set $f_{k}^{rec} = f(x_k)$ and $x_{k}^{rec}= x_k,$ else set $f_{k}^{rec} = f_{k-1}^{rec}$ and $x_{k}^{rec}=x_{k-1}^{rec}$
	\item[Step 2.] If $0\in \partial f(x_k)$, then {\bf stop}.
	\item[Step 3.] If $f(x_k) \leq f_{k(\ell)}^{rec} - \frac{1}{2} \delta_\ell $, set $k(\ell+1) = k, \sigma_k = 0, \delta_{\ell+1} = \delta_\ell$, replace $\ell$ by $\ell+1,$ and go to Step~5.
	\item[Step 4.] If $\sigma_k > R$, set $k(\ell+1) = k, \sigma_k = 0, \delta_{\ell+1} = \frac{1}{2} \delta_\ell$,  $x_k=x_k^{rec}$,  and  $\ell\leftarrow\ell+1$.
	\item[Step 5.] Set $f_k^{lev} := f_{k(\ell)}^{rec} - \delta_\ell$. Select $\beta_k \in [\underline{\beta}, \bar{\beta}]$ and calculate $x_{k+1}$ via Algorithm \ref{Alg:INP} with the stepsize given by Rule \ref{Dynamic.Step}.
	\item[Step 6.] Set $\sigma_{k+1}:= \sigma_k + \tilde{t}_k$,  $k\leftarrow k+1$,  and go to Step 1.
\end{description}
\caption{Subgradient-InexP employing the dynamic stepsize rule (SInexPD)}
\label{Alg:INPDyn}
\end{algorithm}

\noindent

Following  \cite{GoffinKiwiel1999, nedic_bertsekas2001rate, nedic_bertsekas2001, xmwang2018},  we describe the main features of the Subgradient-InexP method.
\begin{remark} \label{eq:bfkrec}
Note that in Step 1,  $f_{k}^{rec}$ keeps the record of the smallest  functional value attained by the iterates generated so far, i.e., $f_{k}^{rec}:=\min\{f(x_j):~j=0, \ldots,k\}$.  Splitting  the iterations into groups
$$
K_\ell := \{k(\ell), k(\ell) + 1, \ldots, k(\ell+1)-1\}, \quad \ell = 0, 1, \dots, 
$$
  Algorithm \ref{Alg:INPDyn} uses the same target level $f_k^{lev} = f_{k(\ell)}^{rec} - \delta_\ell$, for $k \in K_\ell$. Also, note that the target level is update only if sufficient descent or oscillation is detected (Step 3 or Step 4, respectively). Whenever $\sigma_k$ exceeds the upper bound $R$, the parameter $\delta_\ell$ is decreased, which increases the target level $f_k^{lev}$.
\end{remark}
From now on, we assume that Algorithm \ref{Alg:INPDyn} generates an infinite sequence.  In the next theorem we present the result about the asymptotic  convergence of the sequence  $\{x_k\}$. It is the versions of \cite[Theorem~1]{GoffinKiwiel1999} and  \cite[Proposition~2.7]{nedic_bertsekas2001} by using inexact projections. 
\begin{theorem} \label{teo.MainConvDynamic}
	There holds  $\inf_{k \geq 0} f(x_k)= f^*$.
\end{theorem}
\begin{proof}
Since $x_k \in C$ and $x_{k+1} \in {\cal P}_C\left(\varphi_{\gamma_k, \theta_k, \lambda_k}, x_k, x_k - t_ks_k\right)$, by the first equality in \eqref{eq.dynamicstep} and applying  Lemma~\ref{pr:cond} with $w=x_{k+1}$, $v=x_k - t_ks_k$, $x=x_k$ ,  $u=x_k$,   and $\varphi_{ \gamma, \theta, \lambda}=\varphi_{ \gamma_k, \theta_k, \lambda_k}$,  we conclude that 
	\begin{equation*}
		\|x_{k+1}-x_{k}\| \leq \sqrt{\frac{1+2 \bar{\gamma}}{1-2 \bar{\lambda}}} \,\,t_k\|s_k\| \leq \sqrt{\frac{1+2 \bar{\gamma}}{1-2 \bar{\lambda}}} \,\, \tilde{t}_k, \qquad k=0, 1, \ldots.
	\end{equation*}
We  claim that  the index $\ell$ goes to $+\infty$ and either $\inf_{k \geq 0} f(x_k) =-\infty$ or $\lim_{l\to\infty} \delta_\ell = 0$. Indeed, assume that $\ell$ takes only a finite number of values, i.e., $\ell < \infty$. Since $\sigma_k + \tilde{t}_k = \sigma_{k+1} \leq R$, for all $k \geq k(\ell)$, then    we conclude that
	\begin{equation} \label{eq.boundxk}
		\|x_k - x_{k(\ell)}\| \leq \sum_{j=k(\ell)}^k \|x_{j+1} - x_j\| \leq \sqrt{\frac{1+2 \bar{\gamma}}{1-2 \bar{\lambda}}} \sum_{j=k(\ell)}^k \tilde{t}_j= \sqrt{\frac{1+2 \bar{\gamma}}{1-2 \bar{\lambda}}} \,\, \sigma_{k+1} \leq \sqrt{\frac{1+2 \bar{\gamma}}{1-2 \bar{\lambda}}} \,\, R.
	\end{equation}
	Hence, $\{x_k\}$ is bounded. Besides, from the last condition in \eqref{betaDinamic},  the sequence $\{\epsilon_k\}$ is bounded and, by  using Proposition~\ref{pr:CompE-subdif}, $\{s_k\}$ is also bounded. Moreover, by \eqref{eq.boundxk} we also conclude  $\sum_{j=k(\ell)}^{+\infty} \tilde{t}_j<+\infty$, which implies   $\lim_{k\to \infty}  \tilde{t}_k = 0$. Thus, due to  $\beta_k \in [\underline{\beta}, \bar{\beta}]$, it follows from second  equality in \eqref{eq.dynamicstep} that 
	\begin{equation} \label{eq.boundfk}
		\lim_{k \to \infty} [f(x_k) - f_k^{lev}] = 0.
	\end{equation}
	 On the other hand,  Steps 3 and 5 of Algorithm~\ref{Alg:INPDyn} yield 
	$$
	f(x_k) > f_{k(\ell)}^{rec} - \frac{1}{2} \delta_\ell = f_k^{lev} + \delta_\ell - \frac{1}{2} \delta_\ell = f_k^{lev} + \frac{1}{2} \delta_\ell \qquad \quad k= k(\ell), k(\ell)+1,  \ldots, 
	$$
 contradicting \eqref{eq.boundfk}. Therefore, $\ell$ goes to $+ \infty$. Now, suppose that $ \lim_{\ell\to\infty} \delta_\ell=\delta  > 0$. Then,  from Steps 3 and 4  of Algorithm~\ref{Alg:INPDyn},  it follows that for all $\ell$ large enough, we have $\delta_\ell = \delta$ and
	$
	f_{k(\ell+1)}^{rec}  \leq f_{k(\ell)}^{rec} -\frac{1}{2} \delta,
	$
	implying that $\displaystyle\inf_{k \geq 0} f(x_k) = -\infty$, which concludes the claim.   If $\lim_{\ell\to\infty} \delta_\ell > 0$ then, according to above claim, we have $\inf_{k \geq 0} f(x_k) = -\infty$,  obtain the desired result.  Now,  we assume by contradiction that $\lim_{\ell\to\infty} \delta_\ell = 0$ and $\inf_{k \geq 0} f(x_k) > f^*$.   Thus, it follows from Remark~\ref{eq:bfkrec} that 
	$\inf_{k \geq 0} f_{k}^{rec} =\inf_{k \geq 0} f(x_k) $. Hence, we conclude that $\inf_{k \geq 0} f_{k}^{rec}> f^*$. In this case, by using the definition of $\{f_k^{lev}\}$ in Step 5 and taking into account that  $\lim_{\ell\to\infty} \delta_\ell = 0$, we conclude that 
$$
		\displaystyle\inf_{k \geq 0} f_k^{lev} =\displaystyle\inf_{\ell \geq 0} (f_{k(\ell)}^{rec} - \delta_\ell) = \displaystyle\inf_{\ell \geq 0}  f_{k(\ell)}^{rec} > f^*.
$$
	Therefore,  there exist  $\bar{\delta}>0$, ${\bar x}\in  C$ and $\bar{k} \in \mathbb{N}$ such that
	\begin{equation} \label{eq:ahc}
		f_k^{lev} - f({\bar x}) \geq \bar{\delta}, \qquad  k =\bar{k}, \bar{k}+1, \ldots.
	\end{equation}	
	Hence, by using the definition of $\tilde{t}_k$ in \eqref{eq.dynamicstep}, it follows from \eqref{eq:ahc} that
	\begin{equation} \label{eq:limtildtk}
		\tilde{t}_k = \beta_k\frac{f(x_k)-f_{k}^{lev}}{\left\|s_k\right\|} < \bar{\beta} \frac{f(x_k)-f({\bar x})-\bar{\delta}}{\left\|s_k\right\|}, \qquad  k =\bar{k}, \bar{k}+1, \ldots.
	\end{equation}	
	Now, applying Lemma \ref{Le:FejerConv} with $x={\bar x}$ and then using \eqref{StepsizePolyak} and \eqref{eq.dynamicstep}, we obtain 
	$$
	\|x_{k+1} - {\bar x}\|^2  \leq \|x_k-{\bar x}\|^2+ \tilde{t}_k \left(\nu \tilde{t}_k - \frac{2}{\|s_k\|} \left[f(x_k)-f({\bar x})-\epsilon_k\right] \right), \qquad  k =\bar{k}, \bar{k}+1, \ldots.
	$$	
	Thus, the combination of the last inequality with  \eqref{eq:ahc}, \eqref{eq:limtildtk},  and the last inequality in \eqref{betaDinamic} yields
	$$
	\|x_{k+1} - {\bar x}\|^2  \leq \|x_k-{\bar x}\|^2+ \frac{\tilde{t}_k}{\|s_k\|} \bigg(\left[(2\mu+\nu)\bar{\beta}-2\right]\left[f(x_k)-f({\bar x})\right] - \left( 2\mu+\nu\right)\bar{\beta}\bar{\delta}\bigg), \qquad k =\bar{k}, \bar{k}+1, \ldots.
	$$
It follows from \eqref{betaDinamic} that $\bar{\beta} < 2/(2\mu+\nu)$, which implies that $(2\mu+\nu)\bar{\beta}-2 < 0$. Thus, by using that $f(x_k)\geq f_k^{lev}$ for all $k=0, 1, \ldots$, the last inequality implies 
	\begin{equation} \label{eq:xk_fejerdynamic}
		\|x_{k+1} - {\bar x}\|^2 \leq \|x_k-{\bar x}\|^2 - \frac{\tilde{t}_k}{\|s_k\|} \left( 2\mu+\nu\right)\bar{\beta}\bar{\delta}, \qquad k =\bar{k}, \bar{k}+1, \ldots.
	\end{equation}
	Hence, $\|x_{k+1}-{\bar x}\| \leq \|x_{\bar k}-{\bar x}\|$, for all $k \geq \bar{k}$, which implies that $\{x_k\}$ is bounded. Besides, by using \eqref{eq.boundfk}, it follows from the last condition in \eqref{betaDinamic} that the sequence $\left\lbrace \epsilon_k \right\rbrace$ is also bounded. Thus,  using Proposition~\ref{pr:CompE-subdif}, we conclude that there exists $c>0$ such that $\left\|s_k\right\| \leq c$, for all $k \geq 0$, which together   \eqref{eq:xk_fejerdynamic},  yield
$$
\frac{{\bar \beta}\bar{\delta}}{c} \left(2\mu+\nu\right) \displaystyle\sum_{k = \bar{k}}^\infty \tilde{t}_k \leq \|x_{\bar{k}} - \bar{x}\|^2<+\infty.
$$
	Since $\sigma_{k} = \sum_{j= {k(\ell)}}^{ {k(\ell+1)}-1} \tilde{t}_j $, the last inequality implies that there exists $\ell_0 \in \mathbb{N}$ such that 
$$
\sigma_{{k(\ell+1)}} \leq \displaystyle\sum_{k = {k(\ell)}}^\infty \tilde{t}_k < R, \qquad \quad ~\ell = \ell_0, \ell_0+1\ldots .
$$
Hence, Step 4 in  Algorithm~\ref{Alg:INPDyn} cannot occur infinitely to decrease $\delta_\ell$,  contradicting the fact that  $\displaystyle\lim_{\ell\to\infty} \delta_\ell = 0$. Therefore, the result  follows and the proof is concluded.  
\end{proof}
The next result presents an iteration-complexity bound for the subgradient-InexP method with the stepsize given by Rule~\ref{Dynamic.Step}, which  is a version of \cite[Proposition 2.15]{nedic_bertsekas2001rate} for our algorithm.
\begin{theorem}
	Assume that the sequence $\{x_k\}$ converges to a point $x_*\in\Omega^*$.   Let $ \delta_0>0$ be given in  Algorithm \ref{Alg:INPDyn} and $c\geq \max\{\|s_k\|:~k=0, 1,\ldots \}$. Then, 
	\begin{equation} \label{eq:minfxkdynamic}
		\min \{f(x_k) - f^{*}:~ \, k = 0, 1,\ldots, N \}  \leq \delta_0, 
	\end{equation}
	where $N$ is the largest positive integer such that
	\begin{equation} \label{eq:defNdynamic}
		\sum_{k=0}^{N-1}\left( \beta_k \left[2 - (2 \mu+\nu) \beta_k\right]\delta_k^2\right) \leq  \left(c\|x_0 - x_*\|\right)^2.
	\end{equation}
\end{theorem}

\begin{proof}
	Assume by contradiction  that \eqref{eq:minfxkdynamic} does not holds. Thus, for all $k$ with $0 \leq k \leq N$ we have
	$
	f(x_k) > f^* + \delta_0.
	$
	Hence, considering that  $\delta_\ell \leq \delta_0$ for all $\ell$,   we have
	\begin{equation} \label{eq:fklevfstar}
		f_k^{lev} = f_{k(\ell)}^{rec}-\delta_\ell > f^* + \delta_0 - \delta_\ell \geq f^*,  \quad  \qquad k=0, \ldots, N.
	\end{equation}
	The combination of  the last inequality in \eqref{betaDinamic} with \eqref{eq.dynamicstep} gives $2t_k\epsilon_k \leq 2\mu \beta_k^2 \left[f(x_k)-f_k^{lev}\right]^2 / \|s_k\|^2$. Moreover, \eqref{eq.dynamicstep}  implies that $t_k^2\|s_k\|^2 = \beta_k^2 \left[f(x_k)-f_k^{lev}\right]^2/\|s_k\|^2$. Now, using \eqref{eq:fklevfstar}, Lemma~\ref{Le:FejerConv}  with  $x = x_* \in \Omega^*$,  and since $\beta_k \in [\underline{\beta}, \bar{\beta}]$, we obtain
\begin{equation} \label{eq:cdinq}
	\|x_{k+1}-x_*\|^2 \leq \|x_k-x_*\|^2 - \beta_k \left[2 -(2 \mu+\nu) \beta_k\right]\frac{\left[f(x_k)-f_k^{lev}\right]^2}{\|s_k\|^2}.
\end{equation}
Since  $\{x_k\}$ converges to $x_*\in\Omega^*$, Proposition~\ref{pr:CompE-subdif} implies  that there exists $c>0$ such that $\left\|s_k\right\| \leq c$, for $k=0,1,\ldots$. Furthermore, using the fact  $f(x_k)-f_k^{lev} \geq \delta_k$, $0 \leq k \leq N$, \eqref{eq:cdinq} yields 
	$$\
	\|x_{k+1}-x_*\|^2 \leq \|x_k-x_*\|^2 - \beta_k \left[2 -(2 \mu+\nu) \beta_k\right] \frac{\delta_k^2}{c^2}.
	$$
	Performing the sum of the above inequality for $k = 0, 1, \ldots, N,$ we conclude that
	$$
	\sum_{k=0}^{N} \left(\beta_k \left[2 -(2 \mu+\nu) \beta_k\right]\frac{\delta_k^2}{c^2}\right) \leq \|x_0 - x_*\|^2, 
	$$
which contradicts  \eqref{eq:defNdynamic}.
\end{proof}

\section{Numerical results} \label{Sec:NumExp}

Our intention in this section is to report some numerical results in order to illustrate the practical behavior of SInexPD Algorithm when $C$ is a compact convex set. 
We implemented SInexPD Algorithm in Fortran~90 considering set $C$ in the general form $C= \left\{x\in\mathbb{R}^n:~ h(x)=0, g(x)\leq 0 \right\}$, where $h: \mathbb{R}^n \to \mathbb{R}^m $ and $g: \mathbb{R}^n \to \mathbb{R}^p$ are smooth functions. At each iteration $k$, the  Frank-Wolfe algorithm  is used to compute a feasible inexact projection as explained below.  The algorithm codes are freely available at \url{https://orizon.ime.ufg.br/}.

\subsection{Frank-Wolfe algorithm  to find an approximated projection} \label{Sec:CondG}
In this section we use the {\it Frank-Wolfe algorithm}  also known   {\it conditional gradient method}   to find an inexact projection onto a  compact convex set $C\subset \mathbb{R}^n$; papers dealing with this method  include  \cite{BeckTeboulle2004, FrankWolfe1956, JAGGI2013, Konnov2018, LanZhou2016, Ravi2017}.     The exact  projection  of $v\in  \mathbb{R}^n$ onto $C$ is  the solution of the following convex quadratic  optimization problem 
\begin{equation} \label{eq:ProbCond}
	{\min}_{w \in C} \psi(w) := \frac{1}{2}\|w-v\|^2.
\end{equation}
Assume that  $v\notin C$.  Let us  describe the subroutine, which we nominate {\it FW-Procedure},  for finding  an approximated solution of \eqref{eq:ProbCond}  relative to a point  $u \in C$, i.e., a point belonging  to the set ${\cal P}_C(\varphi_{\gamma, \theta, \lambda}, u, v)$,  where the  error tolerance    mapping $\varphi_{\gamma, \theta, \lambda}$ and  the set valued mapping  ${\cal P}_C(\varphi_{\gamma, \theta, \lambda}, u, .)$ are  given  in  Definition~\ref{def:InexactProj}. 
 \begin{algorithm}
\begin{description}
\item[ Step 0.] Set $w_1 = u$ and $k=1$.
\item[ Step 1.] Call the linear optimization oracle (or simply LO oracle) to compute
	\begin{equation}\label{eq:condG}
		z_k := \arg\min_{z \in C} \langle w_k-v, z-w_k \rangle, \qquad g_k^*:= \langle w_k - v, z_k-w_k \rangle.
	\end{equation}
\item[Step 2.] If $g^*_k \geq  - \varphi_{\gamma, \theta, \lambda}(u, v, w_k)$, set $w_{+}:=w_k$ and {\bf stop}; otherwise, compute
	\begin{equation}\label{eq:stepsize}
		\tau_k: = \min\left\{1, \frac{-g^*_k}{\|z_k-w_k\|^2} \right\}, \qquad w_{k+1}:=w_k + \tau_k(z_k-w_k).
	\end{equation}
\item[Step 3.] Set $k \gets k+1$, and go to {\bf Step 1}.
\end{description}
\NoCaptionOfAlgo
\caption{ {\bf FW-Procedure to compute} $w_{+}\in {\cal P}_C(\varphi_{\gamma, \theta, \lambda}, u, v)$}
\end{algorithm}

Since $\psi$ is strictly  convex, we conclude from \eqref{eq:condG} that $\psi(z) > \psi(w_k) + g_k^*$, for all $z \in C$ such that  $z\neq w_k$. Setting  $\psi^*:=\min_{w \in C} \psi(w)$   we have $\psi(w_k) \geq \psi^* \geq \psi(w_k) + g_k^*$, which implies   $g_k^* <0.$  Thus,   the  stepsize $\tau_k$ given  by  \eqref{eq:stepsize} is computed  using exact minimization, i.e., $0<\tau_k := \arg\min_{\tau \in [0,1]} \psi(w_k + \tau(z_k - w_k))$. Since $C$ is convex and $z_k$, $w_k \in C$, we have from  \eqref{eq:stepsize} that $w_{k+1} \in C$, which implies that  all points generated by  {\it FW-Procedure}   are in $C$.    Moreover,  \eqref{eq:condG} implies that  $g_k^* = \langle w_k - v, z_k - w_k \rangle  \leq  \langle w_k - v, z - w_k \rangle$, for all $z \in C$.  Hence,  if  the  stopping criteria    $g_k^* = \langle w_k - v, z_k - w_k \rangle \geq - \varphi_{\gamma, \theta, \lambda}(u, v, w_k)$  in Step 2 of   {\it FW-Procedure}   is satisfied, then  $ \langle  v-w_k , z - w_k \rangle \leq \varphi_{\gamma, \theta, \lambda}(u, v, w_k)$, for all $z \in C$. Therefore,   from Definition~\ref{def:InexactProj},    we conclude that    $w_{+}=w_k\in {\cal P}_C(\varphi_{\gamma, \theta, \lambda}, u, v)$, i.e., the output of {\it FW-Procedure},   is a feasible inexact projection of $v \in \mathbb{R}^n$ relative to $u \in C$.  Finally,   \cite[Proposition A.2]{BeckTeboulle2004} implies that $\lim_{k\to +\infty} g_k^* =0$. Thus, the stopping criteria  $g_k^* \geq - \varphi_{\gamma, \theta, \lambda}(u, v, w_k)$  in Step 2 of   {\it FW-Procedure}   is satisfied in a finite number of iterations if and only if $\varphi_{\gamma, \theta, \lambda}(u, v, w_k)\neq 0$, for all $k=0, 1, \ldots$.  

The following  theorem is an import result about the convergence rate  of the conditional gradient method applied to problem \eqref{eq:ProbCond}, which its proof can be found in \cite{GarberHazan2015}. For stating the theorem,  we first note that 
\begin{equation}\label{eq:PropPsi}
	\psi(w) - \psi(w^*) \leq \frac{1}{2} \|w-w^*\|, \qquad \forall z \in C;
\end{equation}
see also \cite[Theorem 2.1.8]{Nesterov2004}.

\begin{theorem}  \label{th:fcr}
	Let  $d_C := \max_{z,w \in C} \|z-w\|$ be the diameter of C. For $k \geq 1$,  the iterate $w_k$ of \emph{FW-Procedure } satisfies $\psi(w_k) -\psi(w_*) \leq 8d_{C}^2/k$. Consequently, using \eqref{eq:PropPsi},  we have  $\|w_k - w_*\| \leq 4d_{C}/\sqrt{k}$, for all $k \geq 1$.
\end{theorem}

\subsection{Examples}

  Consider the problem
\begin{equation}\label{numprob}
\min_{x \in C} \, f(x) :=\|x\|_1,
\end{equation}
where $C:=\left\{x\in\mathbb{R}^n \colon x\geq0  \mbox{ and }  (x-\bar{x})^TQ(x-\bar{x})\leq 1 \right\}$ for a given vector $\bar{x}\in\R^n$ and a symmetric positive definite matrix $Q\in\mathbb{R}^{n\times n}$.
Since the $\ell_1$ norm tends to promote sparse solutions, we formulated instances of Problem \eqref{numprob} where there are vectors in $C$ with only one non-null component. Thus we can verify the ability of SInexPD Algorithm to recover sparsity.
Let us describe the main characteristics of the considered instances. Consider the spectral decomposition of $Q$ given by
$$
Q=\sum_{i=1}^n\lambda_iv^i(v^i)^T, 
$$
where $\lambda_1 \geq \ldots   \geq \lambda_{n-1}  > \lambda_n>0$ are the eigenvalues of $Q$ and $\{v_1, v_2,\dots,v_n\}$ is an orthonormal system of corresponding eigenvectors. We assume that there exists $u\in  \mathbb{R}^n_{++}$ such that
\begin{equation}\label{eq:id2}
v_n=u/\|u\|, \quad  \lambda_ n<1/\|u\|^2, \qquad \mbox{and} \qquad \bar{x}=u+ \xi e_n,  
\end{equation}
where $\xi\geq 1/\sqrt{\lambda_n}$ and $e_n\in\R^n$ is such that $e_n=[0,\ldots,0,1]^T$. We claim that $\tilde{x}:= \xi e_n\in C$ and $0\notin C$. Indeed,   using \eqref{eq:id2} we have $\tilde{x}-\bar{x}=-\|u\|v_n$ , which implies 
$$
 (\tilde{x}-\bar{x})^TQ (\tilde{x}-\bar{x})= \|u\|^2 v_n^TQv_n= \|u\|^2\lambda_ n<1,
$$
concluding that $\tilde{x}\in C$. Now note that $0\in C$ if and only if $\bar{x}^TQ\bar{x} \leq 1$. Since 
$$
 \xi\geq \frac{1}{\sqrt{\lambda_n}} > -\langle u, e_n\rangle + \frac{1}{\sqrt{\lambda_n}} >\left(-\langle u, e_n\rangle+\sqrt{\langle u, e_n\rangle^2-(\|u\|^2-1/\lambda_n)} \right)> 0
$$
and $\|u+ \xi e_n\|^2=\xi^2+2\langle u, e_n\rangle \xi +\|u\|^2$, we have 
$$\bar{x}^TQ\bar{x}\geq \lambda_n\|\bar{x}\|^2=\lambda_n\|u+ \xi e_n\|^2>1,$$
implying that $0\notin C$.

 For Problem \eqref{numprob}, given $x\in\R^n$ we can get $s\in\partial f(x)$ by taking
$$
[s]_i := \left\{
\begin{array}{rl}
 -1, & \mbox{ if } [x]_i < 0 \\
  0, & \mbox{ if } [x]_i = 0 \\
  1, & \mbox{ if } [x]_i > 0, \\
\end{array}
\right.
$$
where $[\cdot]_i$ stands for the $i$-th component of the corresponding vector.  For computing the optimal solution $z_k$ at Step~1 of the FW-Procedure, we use the software Algencan~\cite{algencan}, an augmented Lagrangian code for general nonlinear optimization programming. We set $R=\|x_1-x_0\|$ and $\delta_0=\|s_0\|/2$ as suggested in \cite{nedic_bertsekas2001} and \cite{GoffinKiwiel1999}, respectively. 
Our implementation uses the stopping criterion
$$\delta_{\ell}\leq 10^{-3}(1+|f_k^{rec}|),$$
also suggested in \cite{GoffinKiwiel1999}. 
Thus, in Algorithm~\ref{Alg:INP}, we have $\epsilon_k=0$ for all $k$. In our tests, we set $x_0=\bar{x}$ and, for all $k$, $\theta_k=0.25$, $\lambda_k=0.025$, $\gamma_k=0.025$, and defined $\beta_k := 2 (1-2\lambda_k)/(1+2\gamma_k)-10^{-6}$ satisfying \eqref{betaDinamic}. Figure~\ref{fig:Behavior} shows the behavior of SInexPD Algorithm on a  two-dimensional instance of Problem \eqref{numprob}. The hatched region represents set $C$ and only the iterates for which the target level was updated are plotted. As can be seen, the algorithm successfully found the {\it solution} for $\ell = 6$ iterations. 
We point out that the algorithm performed a total of $\ell = 14$  ($k=189$)  iterations until it met the stopping criterion.
The highlight of the figure is that, before finding the solution, the iterates belong to the interior of set $C$. This is mostly due to the fact that SInexPD Algorithm performs inexact projections.

\begin{figure}[H]
\centering
\includegraphics[scale=0.70]{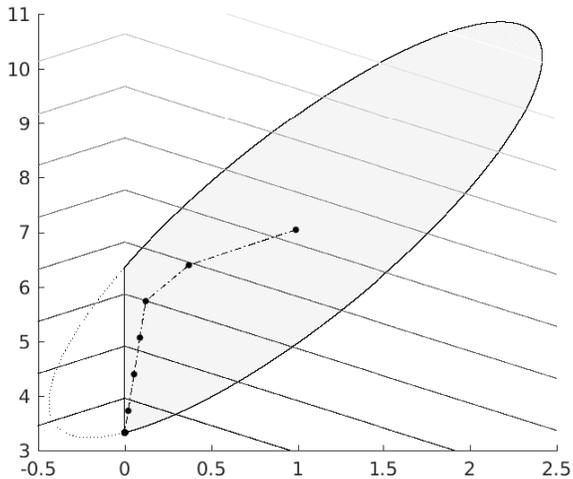}\\
\caption{Behavior of SInexPD Algorithm on a  two-dimensional instance of Problem \eqref{numprob}.}
\label{fig:Behavior}
\end{figure}

Finally, we considered six instances of Problem \eqref{numprob} varying the dimension $n$. Without attempting to go into details, we mention that the problems were randomly generated such that $\lambda_n\in(10^{-2},10^{-6})$, $\lambda_i\in(10,10^{3})$ for $i = 1,\ldots,n-1$, vector $u\in\R^n_{++}$ in \eqref{eq:id2} is such that $\|u\| \in(0.8/\sqrt{\lambda_n},1/\sqrt{\lambda_n})$, and $\xi = 1\sqrt{\lambda_n}$. These imply that, with respect to the ellipsoid that makes up set $C$, the axis corresponding to the eigenvector $v_n$ is {\it much larger} than the others ones. Moreover, the vectors of $C$ that have only one non-null component are {\it far} from the center $\bar{x}$. These characteristics make problems more challenging for the algorithm. Table~\ref{tab:Performance} shows the performance of SInexPD Algorithm. In the table,  column ``$n$" informs the considered dimension, ``$k$" and ``$\ell$" are the number of iterations according to SInexPD Algorithm, ``$\|x_k^{rec}\|_0$" is the number of non-null elements at the final iterate, and  ``$f_k^{rec}$" and ``$\delta_{\ell}$" are their corresponding values at the final iterate.

\begin{table}[h!]
\centering
\begin{tabular}{|c|cc|ccc|}  \hline \rowcolor[gray]{.9}
$n$     &        $k$ &  $\ell$   & $\|x_k^{rec}\|_0$ & $f_k^{rec}$ & $\delta_{\ell}$ \\    \hline  
10  & 91 & 19 & 1 & 1.12D+01 & 6.18D-03 \\
100 & 85 & 21 & 1 & 1.07D+01 & 9.77D-03 \\
200 & 63 & 36 & 1 & 2.11D+01 & 1.38D-02 \\
500 & 58 & 22 & 1 & 1.01D+01 & 1.09D-02 \\
800 & 575 & 26 & 1 & 1.20D+01 & 6.91D-03  \\
1000 & 669 & 24 &  1 & 1.15D+01 & 7.72D-03 \\ \hline 
\end{tabular}
\caption{Performance of SInexPD Algorithm on six instances of Problem \eqref{numprob} varying the dimension.}
\label{tab:Performance}
\end{table}


As showed in Table~\ref{tab:Performance}, the algorithm found vectors with only one non-null component in all instances, showing its ability to recover sparsity in this class of problems.

 Remembering that the table data corresponds to the values when the stop criterion was met, we reported that the {\it final} iterates were found with $\ell = 10, 16, 27, 12, 15$ and $12$ iterations, respectively. We point out that, due to the inexact projections and mimicking the behavior of SInexPD Algorithm in the two-dimensional case, in each instance the iterates remained in the interior of $C$ before the corresponding solution was found.

\section{Conclusions} \label{Sec:Conclusions}
It is well known that the application of the subgradient method is only suitable for certain specific classes of non-differentiable convex optimization problems. However, this method  is basic in the sense that it is the first step towards designing more efficient methods for solving that  problems.   Indeed, it is  intrinsically  related to cutting-plane and bundle methods; see \cite{UrrutyLemarechal1993_II}.  These considerations lead us to conclude that the knowledge of new properties of the subgradient method has great theoretical value. In particular, our inexact version of the projected subgradient method will be useful  in this  theoretical context.  Finally, one issue we believe deserves attention is the  construction of inexact projected  versions of cutting-plane and bundle methods.

\bibliographystyle{habbrv}
\bibliography{SubgradInexP}

\end{document}